\documentclass[12pt,reqno]{amsart}

\usepackage{amssymb}
\usepackage[utf8]{inputenc}

\newtheorem{theorem}{Theorem}[section]
\newtheorem{lemma}[theorem]{Lemma}
\newtheorem{proposition}[theorem]{Proposition}

\numberwithin{equation}{section}

\begin{document}

\baselineskip=15.5pt

\title[Chern scalar curvature and symmetric products]{Chern scalar curvature and symmetric 
products of compact Riemann surfaces}

\author[I. Biswas]{Indranil Biswas}

\address{School of Mathematics, Tata Institute of Fundamental
Research, Homi Bhabha Road, Mumbai 400005, India}

\email{indranil@math.tifr.res.in}

\author[H. Seshadri]{Harish Seshadri}

\address{Indian Institute of Science, Department of Mathematics,
Bangalore 560003, India}

\email{harish@math.iisc.ernet.in}



\date{}

\begin{abstract}
Let $X$ be a compact connected Riemann surface of genus $g \,\ge\, 0$, and let ${\rm Sym}^d(X)$, $d 
\ge 1$, denote the $d$--fold symmetric product of $X$. We show that ${\rm Sym}^d(X)$ admits a 
Hermitian metric with
\begin{enumerate}
\item negative Chern scalar curvature if and only if $g \,\ge\, 2$, and 

\item positive Chern scalar curvature if and
only if $d\, >\, g$.
\end{enumerate}
\end{abstract}

\maketitle

\section{Introduction}

The existence of Riemannian metrics with scalar curvature having a fixed sign on a compact 
manifold has been extensively studied over the last few decades. In particular, one knows that a 
metric with negative scalar curvature (which can, in fact, assumed to be a constant) exists on 
any compact smooth manifold while there are topological obstructions to the existence of metrics 
with positive scalar curvature. In the K\"ahler setting, much of the work has focussed on 
finding an algebraic geometric characterization of projective varieties which admit extremal 
K\"ahler metrics, in particular, K\"ahler-Einstein and constant scalar curvature metrics.

The main result of this paper is a complete characterization of symmetric products of a compact 
Riemann surface, in terms of the number of factors in the product and the genus, which admit 
Hermitian metrics with Chern scalar curvature of a fixed sign:

\begin{theorem}\label{main}
Let $X$ be a compact Riemann surface of genus $g \,\ge\, 0$. For $d \,\ge\, 1$, the $d$--fold
symmetric product ${\rm Sym}^d(X)$ of $X$ admits a Hermitian metric with 
\begin{enumerate}
\item negative Chern scalar curvature if 
and only if $g \,\geq\, 2$, and

\item positive Chern scalar curvature if and
only if $d\, >\, g$.
\end{enumerate}
\end{theorem}

This result is a continuation of our earlier works on the existence of K\"{a}hler metrics with 
holomorphic bisectional curvature \cite{BS1} and Ricci curvature \cite{BS2} with 
fixed sign on symmetric products, and, more generally, Quot schemes associated to Riemann surfaces.

Theorem \ref{main} raises some interesting questions which may have to be tackled by methods 
different from those in this paper:
 
{\it Question 1}: If $d\, >\, g\, \geq\, 0$, does ${\rm Sym}^d(X)$ admit a {\it K\"ahler} metric 
with positive scalar curvature?

{\it Question 2}: If $g \ \ge \ 2$ and $d \ \le \ g $, can ${\rm Sym}^d(X)$ admit a {\it 
Riemannian} metric with positive scalar curvature?

The proof of Theorem \ref{main} is based on a recent result of X. Yang, \cite{Ya}, giving criteria 
for the existence of positive or negative Chern scalar curvature Hermitian metrics on a compact complex 
manifold. These criteria are in terms of the pseudo-effectivity of the canonical or anti-canonical 
line bundles. In effect, we give necessary and sufficient conditions for the pseudo-effectivity (or 
the lack thereof) of the canonical and anti-canonical bundles of a symmetric product of Riemann 
surfaces.
 
\section{Pseudo-effectivity of the canonical and anti-canonical bundles}

\subsection{The canonical line bundle of a symmetric product}

Let $X$ be a compact connected Riemann surface. The genus of $X$ will be denoted by 
$g$. The holomorphic cotangent bundle of $X$ will be denoted by $K_X$.

For any positive integer $d$, let $P_d$ denote the group of permutations of $\{1,\, \cdots,\, d\}$. 
Let $\text{Sym}^d(X)$ be the quotient of $X^d$ by the action of $P_d$ that permutes the factors of 
$X^d$. This $\text{Sym}^d(X)$ is a complex projective manifold of complex dimension $d$. For 
convenience, $\text{Sym}^0(X)$ will denote a point. The holomorphic cotangent bundle of 
$\text{Sym}^d(X)$ will be denoted by $\Omega^1_{\text{Sym}^d(X)}$. For any $1\, \leq\, i\, \leq\, 
d$, let
\begin{equation}\label{pi}
p_i\, :\, X^d\, \longrightarrow\, X
\end{equation}
be the projection to the $i$-th factor. The N\'eron--Severi group $\text{NS}(\text{Sym}^d(X))$ of 
$\text{Sym}^d(X)$ is the subgroup of $H^2(\text{Sym}^d(X),\, {\mathbb Z})$ of Hodge type $(1,1)$. 
We note that $\text{NS}(\text{Sym}^d(X))$ is torsion-free because $H^2(\text{Sym}^d(X),\, {\mathbb 
Z})$ is torsion-free \cite[p.~329, (12.3)]{Ma}.

Let
$$
K_{X,d}\,:=\, K_{\text{Sym}^d(X)}\, =\, \bigwedge\nolimits^d
\Omega^1_{\text{Sym}^d(X)}
$$
be the canonical line bundle of $\text{Sym}^d(X)$. Let
$$
K^{-1}_{X,d}\, :=\, (K_{X,d})^*
$$
be the anti-canonical line bundle of $\text{Sym}^d(X)$.
We will now recall a description of the N\'eron--Severi
class of $K_{X,d}$. Since $\text{Sym}^1(X)\,=\, X$, the class of $K_{X,d}$ is
$2g-2\, \in\, {\mathbb Z}\,=\, \text{NS}(\text{Sym}^1(X))$. Next
assume that $d\, \geq\, 2$. Let $\Delta'$ denote the image of the morphism
$$
\text{Sym}^{d-2}(X)\times X \, \longrightarrow\, \text{Sym}^d(X)\, ,\ \ ((z_1,
\, \cdots,\, z_{d-2}),\, z)\, \longmapsto\, (z_1,\, \cdots,\, z_{d-2},\, z,\, z)\,.
$$
The N\'eron--Severi class of the holomorphic line bundle ${\mathcal O}_{\text{Sym}^d(X)}(\Delta')$
on $\text{Sym}^d(X)$ defined by the divisor $\Delta'$ will be denoted by $\Delta$.
For any holomorphic line bundle $\xi$ on $X$, consider the holomorphic line bundle
$$
\widehat{\xi}\, :=\, \bigotimes_{i=1}^d p^*_i \xi
$$
on $X^d$, where $p_i$ is the projection in \eqref{pi}. The action $P_d$ on $X^d$ lifts to
$\widehat{\xi}$; for any $z\, \in\, X^d$, the action on the fiber $\widehat{\xi}_z$ of the isotropy
subgroup for $z$ is trivial. Hence $\widehat{\xi}$ descends to a holomorphic line bundle on
$\text{Sym}^d(X)$. The element in $\text{NS}(\text{Sym}^d(X))$ corresponding this holomorphic
line bundle will be denoted by ${\mathcal L}_\xi$. This class ${\mathcal L}_\xi$ can
also be constructed as follows. For any $y\, \in\, X$, let $X_y\, \subset\, \text{Sym}^d(X)$
be the image of the morphism
$$
\text{Sym}^{d-1}(X)\, \longrightarrow\, \text{Sym}^d(X)\, ,\ \
(z_1, \, \cdots,\, z_{d-1})\, \longmapsto\, (z_1,\, \cdots,\, z_{d-1},\, y)\, .
$$
Now define
\begin{equation}\label{e0}
{\mathcal L}_{{\mathcal O}(y)}\,:=\, [{\mathcal O}_{\text{Sym}^d(X)}(X_y)]\, \in\,
\text{NS}(\text{Sym}^d(X))
\end{equation}
(see \cite[p.~18]{Ko2}). Note that ${\mathcal L}_{{\mathcal O}(y)}$ is independent of
the choice of $y$.

For an
effective divisor $D\,=\, \sum_{i=1}^n y_i$ on $X$, define ${\mathcal L}_{{\mathcal O}(D)}
\,=\, \sum_{i=1}^n L_{{\mathcal O}(y_i)}$. Finally, for effective divisors $D$, $D'$ on $X$,
define ${\mathcal L}_{{\mathcal O}(D-D')}
\,=\, {\mathcal L}_{{\mathcal O}(D)}-{\mathcal L}_{{\mathcal O}(D')}$ (see \cite[p.~18]{Ko2}).
Now we have
\begin{equation}\label{e1}
K_{X,d}\,=\, {\mathcal L}_{K_X}- \frac{1}{2}\Delta \, \in\, \text{NS}(\text{Sym}^d(X))
\end{equation}
\cite[p.~19, Proposition~2.6]{Ko2}.

A divisor $D$ on $\text{Sym}^d(X)$ is called effective if $H^0(\text{Sym}^d(X),\,
{\mathcal O}_{\text{Sym}^d(X)}(D))\,\not=\, 0$.
The pseudo-effective cone of $\text{Sym}^d(X)$
is the closure of the convex cone in $\text{NS}(\text{Sym}^d(X))_{\mathbb R}\,=\,
\text{NS}(\text{Sym}^d(X))\otimes_{\mathbb Z}\mathbb R$ generated by the
effective divisors.

\subsection{Hyperbolic Riemann surfaces}

\begin{proposition}\label{prop1}
Assume that $g\, \geq\, 2$. Then $K^{-1}_{X,d}$ is not pseudo-effective for any $d$.
\end{proposition}

\begin{proof}
Let
\begin{equation}\label{phi}
\phi\, :\, \text{Sym}^d(X)\, \longrightarrow\, \text{Pic}^d(X)
\end{equation}
be the morphism defined by $(z_1, \, \cdots,\, z_d)\, \longmapsto\,{\mathcal O}_X(\sum_{i=1}^dz_i)$.
For a theta line bundle $\theta_d$ on $\text{Pic}^d(X)$, let
$$
\Theta\,\in\, \text{NS}(\text{Sym}^d(X))
$$
be the N\'eron--Severi class of the pullback $\phi^*\theta_d$. We have
\begin{equation}\label{e2}
\frac{1}{2}\Delta\,=\, (d+g-1){\mathcal L}_{{\mathcal O}(y)} -\Theta\, ,
\end{equation}
where ${\mathcal L}_{{\mathcal O}(y)}$ is the class in \eqref{e0} \cite[p.~1119, Lemma~2.1(i)]{Pa};
we note that ${\mathcal L}_{{\mathcal O}(y)}$ is denoted by $x$ in both \cite{Pa} and \cite{Ko1}.

Note that ${\mathcal L}_{K_X}\,=\, (2g-2){\mathcal L}_{{\mathcal O}(y)}$, because
$\text{degree}(K_X)\,=\, 2g-2$. Therefore, combining \eqref{e1} and \eqref{e2}, we have
$$
K_{X,d}\,=\, (g-d-1){\mathcal L}_{{\mathcal O}(y)} +\Theta\, .
$$
This implies that
\begin{equation}\label{e4}
K^{-1}_{X,d}\,=\, (d-g+1){\mathcal L}_{{\mathcal O}(y)} -\Theta\, .
\end{equation}
Since $d-g+1\, <\, d+g-1$, the point $(-1,\, d-g+1)\, \in\, {\mathbb R}^2$
lies in the open sector determined by the two half-lines $t\cdot (-1,\, d+g-1)$ and $t\cdot (0,\,-1)$,
$t\, \geq\, 0$, going anti-clockwise from $t\cdot (-1,\, d+g-1)$. Now from
\cite[p.~124, Theorem~3]{Ko1} and \eqref{e4} it follows that
$K^{-1}_{X,d}$ is not pseudo-effective. To see this consider the plane in
$\text{NS}(\text{Sym}^d(X))_{\mathbb R}$ generated by $\Theta$ and
${\mathcal L}_{{\mathcal O}(y)}$. First note that $\Theta\, \in\,
\text{NS}(\text{Sym}^d(X))_{\mathbb R}$, which corresponds to the point $(1,\, 0)\,\in\,
{\mathbb R}^2$, lies in the complement of the closure of the above sector
bounded by $t\cdot (-1,\, d+g-1)$ and $t\cdot (0,\,-1)$, $t\, \geq\, 0$, while $\Theta$ is
pseudo-effective because a theta line bundle on $\text{Pic}^d(X)$ is ample. The intersection
of the pseudo-effective cone with the plane in $\text{NS}(\text{Sym}^d(X))_{\mathbb R}$
generated by $\Theta$ and ${\mathcal L}_{{\mathcal O}(y)}$ is the cone generated
by the half-line $t\cdot (-1,\, d+g-1)$, $t\,\geq\, 0$, and a half-line in the fourth
quadrant (where coefficient of $\Theta$ is positive and the coefficient of ${\mathcal
L}_{{\mathcal O}(y)}$ is negative) (see \cite[p.~124, Theorem~3]{Ko1} and \cite[Section~7]{Ko1}).
Hence from \eqref{e4} it follows that $K^{-1}_{X,d}$ is not pseudo-effective.
\end{proof}

\begin{proposition}\label{prop2}
Assume that $g\, \geq\, 2$. Then $K_{X,d}$ is pseudo-effective if and only if
$d\, \leq\, g$.
\end{proposition}

\begin{proof}
First assume that $d\, \leq\, g$. In this case, the map $\phi$ in \eqref{phi}
is birational onto its image. The canonical line bundle $\Omega^d_{\text{Pic}^d(X)}$ is generated by
one global section because it is trivial. Hence it follows that a nonzero section of
$\Omega^d_{\text{Pic}^d(X)}$ pulls back to a nonzero section of $K_{X,d}$ by the natural homomorphism
$\phi^* \Omega^d_{\text{Pic}^d(X)}\, \longrightarrow\, K_{X,d}$ constructed using the
differential of $\phi$. Therefore, we conclude that $K_{X,d}$ is pseudo-effective.

Now assume that $d\, \geq\, g+1$. Then there is a nonempty Zariski open subset
$U\, \subset\, \text{Pic}^d(X)$ such that the restriction
$$
\phi_0\, :=\, \phi\vert_V \, :\, V\,:=\, \phi^{-1}(U)\, \stackrel{\phi}{\longrightarrow}\, U
$$
is an algebraic fiber bundle with fiber ${\mathbb C}{\mathbb P}^{d-g}$. For any $z\, \in\, U$, the 
restriction of $K_{X,d}$ to $\phi^{-1}_0(z)\,=\, {\mathbb C}{\mathbb P}^{d-g}$ is isomorphic to 
${\mathcal O}_{{\mathbb C}{\mathbb P}^{d-g}}(g-d-1)$. Since $g-d-1\, <\, 0$, the restriction of 
$K_{X,d}$ to $\phi^{-1}_0(z)$ does not admit any nonzero holomorphic section. In fact this 
restriction is the dual of an ample line bundle on $\phi^{-1}_0(z)$. From this it follows that 
$K_{X,d}$ is not pseudo-effective.
\end{proof}

\subsection{Low genera}

First assume that $g\,=\, 0$. Then $\text{Sym}^d(X)\,=\, {\mathbb C}{\mathbb P}^d$, and hence
$K^{-1}_{X,d}$ is ample. This implies that $K_{X,d}$ is not pseudo-effective, and
$K^{-1}_{X,d}$ is pseudo-effective.

Now assume that $g\,=\,1$.

\begin{lemma}\label{lem1}
The canonical line bundle $K_{X,d}$ is pseudo-effective if and only if $d\,=\,1$.
\end{lemma}

\begin{proof}
If $d\,=\,1$, then $K_{X,d}\,=\, K_X$ is trivial, hence it is pseudo-effective. If $d\,>\, 1$,
then $\phi$ in \eqref{phi} is a holomorphic fiber bundle over $\text{Pic}^d(X)$ with fiber
${\mathbb C}{\mathbb P}^{d-1}$. Now as in the proof of Proposition \ref{prop2} it follows
that $K_{X,d}$ is not pseudo-effective.
\end{proof}

\begin{lemma}\label{lem2}
Assume that $g\,=\,1$. The anti-canonical line bundle $K^{-1}_{X,d}$ is pseudo-effective
for all $d$.
\end{lemma}

\begin{proof}
The point $(d-g+1){\mathcal L}_{{\mathcal O}(y)} -\Theta\, \in\, 
\text{NS}(\text{Sym}^d(X))_{\mathbb R}$ coincides with the point $(d+g-1){\mathcal L}_{{\mathcal 
O}(y)} -\Theta$ if $g\,=\, 1$. Hence from \cite[p.~124, Theorem~3]{Ko1} and \eqref{e4} it follows 
immediately that $K^{-1}_{X,d}$ is pseudo-effective.
\end{proof}

\section{Scalar curvature and symmetric products} 

A Gauduchon metric on $\text{Sym}^d(X)$ is a Hermitian metric $h$ on $\text{Sym}^d(X)$
such that the associated $(1,\, 1)$--form $\omega_h$ satisfies the equation
$$
\partial\overline{\partial} \omega^{d-1}_h\,=\, 0\, .
$$
Take a Gauduchon metric $h$.
The \textit{Chern scalar curvature} of $h$ is the real valued function
$$
s_h\, :=\, \text{tr}_{\omega_h}\text{Ric}(\omega_h)\, ,$$
where $\text{Ric}(\omega_h)$ is the
Ricci curvature for $\omega_h$. The \textit{total scalar curvature} for $h$ is the 
integral
$$
\int_{\text{Sym}^d(X)} s_h \omega^{d}_h\,=\, \int_{\text{Sym}^d(X)}\text{Ric}(\omega_h)
\wedge \omega^{d-1}_h\, .
$$

Let ${\mathcal W}(\text{Sym}^d(X))$ denote the space of all
Gauduchon metrics on $\text{Sym}^d(X)$. Let
$$
{\mathcal F}\, :\, {\mathcal W}(\text{Sym}^d(X))\, \longrightarrow\,
{\mathbb R}\, ,\ \ h\, \longmapsto\, \int_{\text{Sym}^d(X)} s_h \omega^{d}_h
$$
be the map that assigns the total scalar curvature to a Gauduchon metric.

Denote
$$
{\mathbb R}^{>0}\, :=\, \{c\, \in\, {\mathbb R}\, \mid\, c\, >0\}\ \
\text{ and }\ \ {\mathbb R}^{<0}\, =\,- {\mathbb R}^{>0}\, .
$$

\begin{theorem}\label{thm1}
Let $X$ be a compact connected Riemann surface of genus $g\,\geq\, 0$;
take any $d\, \geq\, 1$. Then the following
four hold:
\begin{enumerate}
\item ${\mathcal F}({\mathcal W}({\rm Sym}^d(X)))\,=\, {\mathbb R}$ if and only
if $d\, >\,g\, \geq\, 2$.

\item ${\mathcal F}({\mathcal W}({\rm Sym}^d(X)))\,=\, {\mathbb R}^{<0}$
if $g \ge 2$ and $d \, \leq\, g$.

\item ${\mathcal F}({\mathcal W}({\rm Sym}^d(X)))\,=\, {\mathbb R}^{>0}$ if either 
 $g\, =\, 0$ or $g\, =\, 1 \, <\, d$.

\item ${\mathcal F}({\mathcal W}({\rm Sym}^d(X)))\,=\, 0$ if $g\, =\, 1 \, =\, d$.
\end{enumerate}
\end{theorem}

\begin{proof}
In our context, Theorem 1.1 of \cite{Ya} says the following:
\begin{enumerate}
\item ${\mathcal F}({\mathcal W}({\rm Sym}^d(X)))\,=\,\mathbb R$ if and only if
neither $K_{X,d}$ nor $K^{-1}_{X,d}$ is pseudo-effective;

\item ${\mathcal F}({\mathcal W}({\rm Sym}^d(X)))\,=\,{\mathbb R}^{< 0}$ if and only if
$K_{X,d}$ is pseudo-effective but not unitary flat;

\item ${\mathcal F}({\mathcal W}({\rm Sym}^d(X)))\,=\,{\mathbb R}^{>0}$ if and only if
$K^{-1}_{X,d}$ is pseudo-effective but not unitary flat;

\item ${\mathcal F}({\mathcal W}({\rm Sym}^d(X)))\,=\, 0$ if and only if
$K_{X,d}$ is unitary flat.
\end{enumerate}

First assume that $g\, \geq\, 2$. In view of the above result of \cite{Ya}, from
Proposition \ref{prop1} and Proposition \ref{prop2} it follows that
${\mathcal F}({\mathcal W}({\rm Sym}^d(X)))\,=\, {\mathbb R}$
if and only if $d\,>\, g$. It also follows from Proposition \ref{prop1} and Proposition \ref{prop2}
that ${\mathcal F}({\mathcal W}({\rm Sym}^d(X)))\,=\, {\mathbb R}^{<0}$ if $d\, \leq\, g$.

If $g\,=\, 0$, then ${\mathcal F}({\mathcal W}({\rm Sym}^d(X)))\,=\,{\mathbb R}^{>0}$ for
all $d$, because $K^{-1}_{X,d}$ is ample.

Now assume that $g\,=\,1$. If $d\,=\, 1$, then $K_{X,d}$ is trivial, and hence from the fourth 
statement of the above theorem of \cite{Ya} it follows that ${\mathcal F}({\mathcal W}({\rm 
Sym}^d(X)))\,=\, 0$. Next we observe that $K^{-1}_{X,d}$ is not unitary flat when $d\, >\, g$.
Indeed, in this case $\phi$ in \eqref{phi} is a
holomorphic fiber bundle with fibers ${\mathbb C}{\mathbb P}^{d-1}$.
The restriction of $K^{-1}_{X,d}$ to a fiber of $\phi$ is ample, so
$K^{-1}_{X,d}$ is not unitary flat. Therefore, from Lemma \ref{lem2} and the
third statement of the above theorem of \cite{Ya} we conclude that
${\mathcal F}({\mathcal W}({\rm Sym}^d(X)))\,=\,{\mathbb R}^{>0}$ if $d\, >\,g\,=\, 1$.
\end{proof}

\begin{theorem}
Let $X$ be a compact Riemann surface of genus $g \,\ge\, 0$. For $d \,\ge\, 1$, the $d$-fold
symmetric product ${\rm Sym}^d(X)$ admits a Hermitian metric with 
\begin{enumerate}
\item negative Chern scalar curvature if 
and only if $g \,\ge \, 2$, and 

\item positive Chern scalar curvature if and
only if and $d\, >\, g$.
\end{enumerate}
\end{theorem}

\begin{proof}
In our context, Theorem 1.3 of \cite{Ya} says the following:
\begin{enumerate}
\item ${\rm Sym}^d(X)$ admits a Hermitian metric with negative Chern scalar curvature if
and only if $K^{-1}_{X,d}$ is not pseudo-effective;

\item ${\rm Sym}^d(X)$ admits a Hermitian metric with positive Chern scalar curvature if
and only if $K_{X,d}$ is not pseudo-effective.
\end{enumerate}

First assume that $g\, \geq\, 2$. Now from Proposition \ref{prop1} it follows that
${\rm Sym}^d(X)$ admits a Hermitian metric with negative Chern scalar curvature for every $d$.
From Proposition \ref{prop2} it follows that
${\rm Sym}^d(X)$ admits a Hermitian metric with positive Chern scalar curvature if and
only if $d\, >\, g$.

If $g\,=\, 0$, then ${\rm Sym}^d(X)$ admits a Hermitian metric with positive Chern scalar curvature
for all $d$.

Now assume that $g\,=\, 1$. From Lemma \ref{lem1} it follows that for each $d\, \geq\, 2$, the 
symmetric product ${\rm Sym}^d(X)$ admits a Hermitian metric with positive Chern scalar curvature. On 
the other hand, from Lemma \ref{lem2} it follows that that ${\rm Sym}^d(X)$ does not admit a 
Hermitian metric with negative Chern scalar curvature for all $d\, \geq\, 2$. If $d\,=\,1\,=\, g$,
then there is no Hermitian metric with positive or negative Chern scalar curvature. 
\end{proof}


\end{document}